\documentclass[10pt]{amsart}
\usepackage{enumerate,amsmath,amssymb,latexsym,
amsfonts, amsthm, amscd, MnSymbol}

\theoremstyle{plain}

\newtheorem{theorem}{Theorem}

\numberwithin{equation}{section}

\newcommand{\ra}{\rightarrow}

\addtocounter{section}{-1}

\begin{document}

\title {Conservative inclusion of N\"orlund methods}

\date{}

\author[P.L. Robinson]{P.L. Robinson}

\address{Department of Mathematics \\ University of Florida \\ Gainesville FL 32611  USA }

\email[]{paulr@ufl.edu}

 \keywords{}

\begin{abstract}

We offer practical necessary and sufficient conditions in order that every sequence convergent relative to the N\"orlund summation method $(N, p)$ be convergent relative to the N\"orlund summation method $(N, q)$ without the requirement that limits be preserved. 

\end{abstract}

\maketitle

\medbreak

\section*{Conservative Inclusion} 

Let $(p_n : n \geqslant 0)$ be a real sequence with $p_0 > 0$ and with $p_n \geqslant 0$ whenever $n > 0$; write $P_n = p_0 + \cdots + p_n$ whenever $n \geqslant 0$. The {\it N\"orlund method} $(N, p)$ associates to each sequence $r = (r_n : n \geqslant 0)$ the sequence $N^p r$ defined by the rule that for each $m \geqslant 0$ 
$$(N^p r)_m = \frac{p_0 r_m + \cdots + p_m r_0}{p_0 + \cdots + p_m}.$$
We say that $r$ is $(N, p)$-{\it convergent} to the limit $\sigma$ precisely when $N^p r$ is convergent to $\sigma$ in the ordinary sense. We say that the N\"orlund method $(N, p)$ is {\it regular} precisely when each (ordinarily) convergent sequence is $(N, p)$-convergent to the same limit. 

\medbreak 

Now, let $(N, p)$ and $(N, q)$ be N\"orlund methods. We say that $(N, q)$ {\it includes} $(N, p)$ and write $(N, p) \rightsquigarrow (N, q)$ precisely when each $(N, p)$-convergent sequence is $(N, q)$-convergent to the same limit; as a special case, $(N, q)$ is regular precisely when $(N, u) \rightsquigarrow (N, q)$ where $u_0 = 1$ and $u_n = 0$ whenever $n > 0$. Practical necessary and sufficient conditions for $(N, p) \rightsquigarrow (N, q)$ were discovered by Marcel Riesz and communicated by letter [2] to Hardy, who included them in Chapter IV of his classic `{\it Divergent Series}' [1]. Riesz and Hardy only presented these necessary and sufficient conditions in case the N\"orlund methods $(N, p)$ and $(N, q)$ are regular - arguably the most important case. However, we showed recently [3] that the conditions are in fact always valid, without regularity assumptions. 

\medbreak 

It is of some interest to weaken this notion of inclusion, by simply requiring that each $(N, p)$-convergent sequence is $(N, q)$-convergent but not insisting that limits are preserved; let us symbolize this relation by $(N, p) \nersquigarrow (N, q)$ with an `off-target' arrow. For this weakened relation, the name {\it conservative inclusion} suggests itself. In fact, let $C_{p q} = [c_{m, n} : m, n \geqslant 0]$ be the matrix with the property that $N^q r = C_{p q} (N^p r)$ for each sequence $r$. On the one hand, we shall see that $(N, p) \nersquigarrow (N, q)$ holds precisely when the matrix $C_{p q}$ is conservative in the sense that multiplication by $C_{p q}$ sends convergent sequences to convergent sequences but may change limits. On the other hand, we saw in [3] that $(N, p) \rightsquigarrow (N, q)$ holds precisely when $C_{p q}$ is regular in the sense that multiplication by $C_{p q}$ sends convergent sequences to convergent sequences with preservation of limits; thus `regular' inclusion may be called {\it regular inclusion}. 

\medbreak 

Our aim here is to present practical necessary and sufficient conditions (similar to those found by Riesz) for the conservative inclusion of N\"orlund methods. Throughout, we shall use freely the notation of [3]. 

\medbreak 

\section*{Riesz Conditions} 

\medbreak 

Let $C = [ c_{m, n} : m, n \geqslant 0]$ be an infinite square matrix. To each sequence $s = (s_n : n \geqslant 0)$ we associate the sequence $t = C s$ defined by the rule that for each $m \geqslant 0$ 
$$t_m := \sum_{n = 0}^{\infty} c_{m, n} s_n$$
assumed convergent; in our application to N\"orlund methods, $C$ will be lower triangular so that convergence is not an issue. The matrix $C$ is said to be: {\it conservative} precisely when the convergence of $s$ implies the convergence of $C s$ but need not imply that $\lim Cs = \lim s$; {\it regular} precisely when conservative and limit-preserving. 

\medbreak 

A theorem of Kojima and Schur gives practical necessary and sufficient conditions for the matrix $C$ to be conservative. 

\medbreak 

\begin{theorem} \label{KS} 
The matrix $C$ is conservative precisely when each of the following conditions is satisfied: \\
{\rm (i)} there exists $H \geqslant 0$ such that for each $m \geqslant 0$ 
$$\sum_{n = 0}^{\infty} |c_{m, n}| \leqslant H;$$\\
{\rm (ii)} for each $n \geqslant 0$ there exists the limit 
$$\delta_n = \lim_{m \ra \infty} c_{m, n};$$\\
{\rm (iii)} there exists the limit 
$$\delta = \lim_{m \ra \infty} \sum_{n = 0}^{\infty} c_{m, n}.$$
\end{theorem} 

\begin{proof} 
See [1] Theorem 1. 
\end{proof} 

\medbreak 

See [1] Theorem 2 for the corresponding characterization of regular matrices, in which $\delta = 1$ and $\delta_n = 0$ whenever $n \geqslant 0$. 

\medbreak 

Let $(N, p)$ and $(N, q)$ be N\"orlund methods. As $p_0$ is nonzero, the triangular Toeplitz system 
$$q_n = k_0 p_n + \cdots + k_n p_0 \; \; \; (n \geqslant 0)$$
has a unique solution $(k_n : n \geqslant 0)$; by summation, this solution also satisfies 
$$Q_n = k_0 P_n + \cdots + k_n P_0 \; \; \; (n \geqslant 0).$$

\medbreak 

The following result exhibits the square matrix $C_{p q}$ with the property that if $r$ is any sequence then $N^q r = C_{p q} (N^p r)$. 

\medbreak 

\begin{theorem} \label{NN} 
If $r = (r_n : n \geqslant 0)$ is any sequence then 
$$ (N^q r)_m = \sum_{n = 0}^{\infty} c_{m, n} (N^p r)_n$$
where if $n > m$ then $c_{m, n} = 0$ while if $n \leqslant m$ then $c_{m, n} = k_{m - n} P_n / Q_m$. 
\end{theorem} 

\begin{proof} 
This is Theorem 2 of [3]. 
\end{proof} 

\medbreak 

As indicated, we shall henceforth denote the square matrix $[ c_{m, n} : m, n \geqslant 0]$ displayed here by $C_{p q}$. 

\medbreak 

In preparation for our next result, notice that if $s = (s_n : n \geqslant 0)$ is any sequence then $s = N^p r$ for a unique sequence $r$ found by solving the triangular Toeplitz system 
$$P_n s_n = p_0 r_n + \cdots + p_n r_0 \; \; \; (n \geqslant 0).$$

\medbreak 

\begin{theorem} \label{squig}
The conservative inclusion $(N, p) \nersquigarrow (N, q)$ holds precisely when the matrix $C_{p q}$ is conservative. 
\end{theorem} 

\begin{proof} 
Assume $(N, p) \nersquigarrow (N, q)$. Let the sequence $s$ be convergent: with $s = N^p r$ as discussed prior to the theorem, $r$ is $(N, p)$-convergent so that (by assumption) $r$ is $(N, q)$-convergent; now Theorem \ref{NN} tells us that $C_{p q} s = C_{p q} (N^p r) = N^q r$ is convergent. This proves $C_{p q}$ conservative. 
\medbreak 
Assume that $C_{p q}$ is conservative. If $r$ is $(N, p)$-convergent, then $N^p r$ converges whence (by assumption and Theorem \ref{NN}) $N^q r$ converges, so $r$ is $(N, q)$-convergent. This proves that $(N, p) \nersquigarrow (N, q)$. 
\end{proof} 

\medbreak 

We now summon Theorem \ref{KS} to derive practical conditions for the specific matrix $C_{p q}$ to be conservative. 

\medbreak 

\begin{theorem} \label{Riesz} 
The matrix $C_{p q}$ is conservative precisely when each of the following conditions is satisfied: \\
{\rm (1)} there exists $H \geqslant 0$ such that for each $m \geqslant 0$  
$$|k_0| P_m + \cdots + |k_m| P_0 \leqslant H Q_m;$$
{\rm (2)} for each $n \geqslant 0$ there exists the limit 
$$\varepsilon_n = \lim_{m \ra \infty} \frac{k_{m - n}}{Q_m}.$$
\end{theorem} 

\begin{proof} 
Theorem \ref{KS}(i) is equivalent to the existence of $H \geqslant 0$ such that 
$$\frac{|k_m| P_0 + \cdots + |k_0| P_m}{Q_m} =  \sum_{n = 0}^{m} \Big|\frac{k_{m - n} P_n}{Q_m}\Big| = \sum_{n = 0}^{\infty} |c_{m, n}| \leqslant H$$
or 
$$|k_0| P_m + \cdots + |k_m| P_0 \leqslant H Q_m.$$
\medbreak 
\noindent
Theorem \ref{KS}(ii) is equivalent to the existence, for each $n \geqslant 0$, of the limit 
$$\delta_n = \lim_{m \ra \infty} c_{m, n} = \lim_{m \ra \infty} \frac{k_{m - n} P_n}{Q_m};$$
equivalently, of the limit 
$$\varepsilon_n = \lim_{m \ra \infty} \frac{k_{m - n}}{Q_m}.$$
\medbreak 
\noindent 
Theorem \ref{KS}(iii) is automatic (with $\delta = 1$) by virtue of the fact that for each $m \geqslant 0$ 
$$Q_m = k_0 P_m + \cdots + k_m P_0.$$
\end{proof} 

\medbreak 

 Taken together, Theorem \ref{squig} and Theorem \ref{Riesz} yield practical necessary and sufficient conditions for conservative inclusion of N\"orlund methods. 

\medbreak 

We pause for an inspection of these conditions. Condition (1) is exactly the first Riesz condition, labelled ${\bf R}_{p q}^1$ in [3]. Condition (2) bifurcates according to whether the value of $\varepsilon_0$ is or is not zero. 

\medbreak 

The case $0 = \varepsilon_0 = \lim_{m \ra \infty} (k_m / Q_m)$ amounts to the second Riesz condition, labelled ${\bf R}_{p q}^2$ in [3]. In this case, the matrix $C_{p q}$ is regular and in fact $(N, p) \rightsquigarrow (N, q)$ holds; see Theorem 4 and Theorem 3 in [3]. Of course, in this case each $\varepsilon_n$ is zero because 
$$\frac{|k_{m - n}|}{Q_m} \leqslant \frac{|k_{m - n}|}{Q_{m - n}} \ra 0 \; {\rm as} \; m \ra \infty.$$

\medbreak 

The case $\lim_{m \ra \infty} k_m/Q_m = \varepsilon_0 \neq 0$ is new; in this case, if $m$ is large enough then $k_m$ is nonzero (and indeed of constant sign). Fix $n \geqslant 0$: passing to the limit as $m \ra \infty$ in 
$$\frac{Q_{m - n}}{Q_m} = \frac{Q_{m - n}}{k_{m - n}} \frac{k_{m - n}}{Q_m}$$
yields 
$$\lim_{m \ra \infty} \frac{Q_{m - n}}{Q_m} = \frac{\varepsilon_n}{\varepsilon_0}$$
which taken with 
$$\frac{Q_{m - n}}{Q_m} = \frac{Q_{m - n}}{Q_{m - n + 1}} \cdot \; \cdots \; \cdot \frac{Q_{m - 1}}{Q_m}$$
yields 
$$\frac{\varepsilon_n}{\varepsilon_0} = \Big( \frac{\varepsilon_1}{\varepsilon_0} \Big)^n.$$
\medbreak
\noindent
Accordingly, the case $\varepsilon_0 \neq 0$ has only $\varepsilon_1$ as an additional parameter, with 
$$\lim_{m \ra \infty} \frac{k_{m - n}}{Q_m} = \varepsilon_0 \Big(\frac{\varepsilon_1}{\varepsilon_0}\Big)^n.$$ 

\medbreak 

This simplifies the necessary and sufficient conditions as follows. 

\medbreak 

\begin{theorem} \label{necsuf}
The conservative inclusion $(N, p) \nersquigarrow (N, q)$ holds precisely when: \\
{\rm (1)} there exists $H \geqslant 0$ such that $|k_0| P_m + \cdots + |k_m| P_0 \leqslant H Q_m$ whenever  $m \geqslant 0;$ and \\
{\bf either} {\rm (2)(i)} there exist $\varepsilon_0 = \lim_{m \ra \infty} (k_m / Q_m) \neq 0$ and $\beta = \lim_{m \ra \infty} (Q_{m - 1} / Q_m)$ \\ {\bf or} {\rm (2)(ii)} the sequence $(k_m / Q_m : m \geqslant 0)$ converges to zero. 
\end{theorem} 

\begin{proof} 
Simply combine Theorem \ref{squig} and Theorem \ref{Riesz} with the discussion thereafter. Note that in case (i) if $n > 0$ then $k_{m - n} / Q_m \ra \varepsilon_0 \beta^n$ as $m \ra \infty$ while case (ii) corresponds to regular inclusion $(N, p) \rightsquigarrow (N, q)$.  
\end{proof} 

\medbreak 

In closing, we should perhaps clear up a point of potential concern. Recall that any two regular N\"orlund methods $(N, p)$ and $(N, q)$ are {\it consistent} in the sense that if a sequence is both $(N, p)$-convergent to $\sigma$ and $(N, q)$-convergent to $\tau$ then $\sigma = \tau$; see [1] Theorem 17 for example. On account of this, if regular N\"orlund methods satisfy $(N, p) \nersquigarrow (N, q)$ then they actually satisfy $(N, p) \rightsquigarrow (N, q)$. It is comforting to see that this is reflected in our necessary and sufficient conditions. 

\medbreak

Thus, assume that $(N, p)$ and $(N, q)$ satisfy the first Riesz condition ${\bf R}_{p q}^1$ and assume that $k_m / Q_m \ra \varepsilon_0 \neq 0$ as $m \ra \infty$. We shall show that that these assumptions are incompatible with the regularity of $(N, q)$. Recall (from [1] Theorem 17 for instance) that the regularity of $(N, q)$ is equivalent to the requirement that $Q_{m - 1} / Q_m \ra 1$ as $m \ra \infty$. In the present context, this is equivalent to $\varepsilon_1 = \varepsilon_0$ and implies that for each $n \geqslant 0$ 
$$\lim_{m \ra \infty} \frac{k_{m - n}}{Q_m} = \varepsilon_n = \varepsilon_0 \Big( \frac{\varepsilon_1}{\varepsilon_0} \Big)^n = \varepsilon_0 \neq 0$$ 
in light of the simplifying discussion that leads to Theorem \ref{necsuf}. Now, fix 
$$n > \frac{2 H}{P_0 |\varepsilon_0|}\: .$$
For each $0 \leqslant \nu < n$ choose $m_{\nu} \geqslant \nu$ so that 
$$m \geqslant m_{\nu} \Rightarrow \frac{|k_{m - \nu}|}{Q_m} \geqslant \frac{1}{2}|\varepsilon_0| > 0$$ 
and choose any $m$ greater than each of $m_0 , \dots , m_{n - 1}$. From ${\bf R}_{p q}^1$ in the form 
$$|k_0| P_m + \cdots + |k_m| P_0 \leqslant H Q_m$$
we deduce that 
$$(|k_{m - n + 1}| + \cdots + |k_m| ) P_0  \leqslant H Q_m$$
\medbreak
\noindent
and draw the absurd conclusion 
$$n \leqslant \frac{2 H}{P_0 |\varepsilon_0|}\:.$$

\medbreak 

\bigbreak

\begin{center} 
{\small R}{\footnotesize EFERENCES}
\end{center} 
\medbreak

[1] G.H. Hardy, {\it Divergent Series}, Clarendon Press, Oxford (1949). 

\medbreak 

[2] M. Riesz, {\it Sur l'\'equivalence de certaines m\'ethodes de sommation}, Proceedings of the London Mathematical Society (2) {\bf 22} (1924) 412-419. 

\medbreak 

[3] P.L. Robinson, {\it Marcel Riesz on N\"orlund Means}, arXiv 1712.08592 (2017).

\medbreak

\end{document}